\documentclass[11pt]{amsart}
\usepackage{amsmath,amsthm,amssymb,amscd}

\theoremstyle{plain}

\newtheorem{thm}{Theorem}[section]
\newtheorem{proposition}[thm]{Proposition}
\newtheorem{lemma}[thm]{Lemma}

\newtheorem{theorem}[thm]{Theorem}

\newtheorem{lem}[thm]{Lemma}

\newtheorem{prop}[thm]{Proposition}

\theoremstyle{definition}

\newtheorem{rem}[thm]{Remark}

\newcommand{\Tr}{\operatorname{Tr}}
\newcommand{\TR}{\operatorname{TR}}
\newcommand{\tr}{\operatorname{tr}}
\newcommand{\res}{\operatorname{res}}

\allowdisplaybreaks

\begin{document}
\title[Classical and quantum
ergodicity on orbifolds]{Classical and quantum ergodicity on
orbifolds}
\author{Yuri A. Kordyukov}
\address{Institute of Mathematics, Russian Academy of Sciences, 112
Chernyshevsky street, 450008 Ufa, Russia}
\email{yurikor@matem.anrb.ru}

\thanks{The author was partially supported by Grant 10-01-00088-a of the Russian Foundation of
Basic Research and Grant 2010-220-01-077 (contract 11.G34.31.0005)
of the Russian Government}



\begin{abstract}
We extend to orbifolds classical results on quantum ergodicity due
to Shnirelman, Colin de Verdi\`ere and Zelditch, proving that, for
any positive, first-order self-adjoint elliptic pseudodifferential
operator $P$ on a compact orbifold $X$ with positive principal
symbol $p$, ergodicity of the Hamiltonian flow of $p$ implies
quantum ergodicity for the operator $P$. We also prove ergodicity of
the geodesic flow on a compact Riemannian orbifold of negative
sectional curvature.
\end{abstract}
\date{}
 \maketitle


\section{Introduction and main results}

Orbifolds were first introduced in the 1950s by Satake as
topological spaces generalizing smooth manifolds. Since then,
orbifolds became clearly important objects both in mathematics and
in mathematical and theoretical physics. The main purpose of this
paper is to generalize  to orbifolds some basic results on quantum
and classical ergodicity. We refer the reader to \cite{Zelditch10}
for a survey of recent results on quantum ergodicity.

Let $X$ be a compact orbifold, and $P$ a positive, first-order
self-adjoint elliptic pseudodifferential operator on $X$ with
positive principal symbol $p\in S^1(T^*X)$. (Here $S^m(T^*X)$
denotes the space of smooth functions on $T^*X\setminus\{0\}$,
homogeneous of degree $m$ with respect to a fiberwise $\mathbb
R$-action on $T^*X$.) As an example, one can consider
$P=\sqrt{\Delta_X}$, where $\Delta_X$ is the Laplace-Beltrami
operator associated to a Riemannian metric $g_X$ on $X$.

The spectrum of $P$ is discrete, and there is an orthonormal basis
of eigenfunctions $\{\psi_j\}$ with corresponding eigenvalues
$\{\lambda_j\}$:
\[
P\psi_j=\lambda_j\psi_j, \quad \lambda_1\leq \lambda_2\leq\cdots.
\]

Let $f_t$ be the bicharacteristic flow of the operator $P$, that is,
the restriction of the Hamiltonian flow of $p$ to
$S^*_pX=\{(x,\xi)\in T^*X : p(x,\xi)=1\}$. In particular, the
bicharacteristic flow of the operator $P=\sqrt{\Delta_X}$ is the
geodesic flow of the metric $g_X$ on the cosphere bundle
$S^*X=\{(x,\xi)\in T^*X : |\xi|=1\}$.

For any $A\in \Psi^0(X)$, we denote by $\sigma_A\in S^0(T^*X)$ the
principal symbol of $A$. Define a functional $\omega$ on $\Psi^0(X)$
by
\[
\omega(A)=\frac{1}{vol(S^*_pX)}\int_{S^*_pX}\sigma_A\,d\mu,
\]
where $d\mu$ is the Liouville measure on $S^*_pX$.

The next theorem is an orbifold version of the quantum ergodicity
theorem of Shnirelman, Colin de Verdi\`ere, and Zelditch
\cite{Sh,CdV,Zelditch87}.

\begin{thm}\label{t:main1}
If the flow $f_t$ is ergodic on $(S^*_pX,d\mu)$ then there is a
subsequence $\{\psi_{j_k} \}$ of density one such that for any $A\in
\Psi^0(X)$
\[
\lim_{k\to\infty}\langle A\psi_{j_k}, \psi_{j_k}\rangle=\omega(A).
\]
\end{thm}

This theorem can be reformulated in terms of the operator averages
as in \cite{Zelditch96}. For any $A\in \Psi^0(X)$ denote
\[
\langle A\rangle : = \lim_{T\to
\infty}\frac{1}{2T}\int_{-T}^Te^{itP}Ae^{-itP}\,dt,
\]
where the limit is taken in the weak operator topology.

Consider the eigenvalue distribution function $N(\lambda)$ of the
operator $P$:
\[
N(\lambda)=\#\{j:\lambda_j\leq\lambda\}.
\]
Let $E_\lambda$ denote the spectral projection of $P$, corresponding
to the semi-axis $(-\infty,\lambda]$. Thus, we have
\[
N(\lambda)=\operatorname{tr}E_\lambda.
\]

\begin{thm}\label{t:main2}
If the flow $f_t$ is ergodic on $(S^*_pX,d\mu)$ then for any $A\in
\Psi^0(X)$ we have
\[
\langle A\rangle =\omega(A)I+K,
\]
where
\[
\|E_\lambda K E_\lambda\|=o(N(\lambda)), \quad \lambda\to +\infty.
\]
\end{thm}

To provide examples of ergodic geodesic flows on orbifold, we prove
the following extension of a classical result by Anosov
\cite{Anosov67}.

\begin{thm}\label{t:ergodic}
The geodesic flow on a compact Riemannian orbifold of negative
sectional curvature is ergodic.
\end{thm}

As an example of a compact Riemannian orbifold of negative sectional
curvature, one can consider the hyperbolic orbifold
$\mathbb{H}^n/\Gamma$, which is the quotient of the $n$-dimensional
hyperbolic space $\mathbb{H}^n$ by a cocompact discrete group
$\Gamma$ of orientation-preserving isometries of $\mathbb{H}^n$.

Spectral theory of elliptic operators on orbifolds has received much
attention recently (see, for instance, a brief survey in the
introduction of \cite{Dryden-G-G-W}). In \cite{Stanhope-Uribe}, the
Duistermaat-Guillemin trace formula was extended to compact
Riemannian orbifolds. This formula has been applied in
\cite{Guilemin-Uribe-Wang} to an inverse spectral problem on some
orbifolds. In \cite{orbi-Egorov}, we stated two versions of Egorov's
theorem for orbifolds as well as Egorov's theorem for matrix-valued
pseudodifferential operators on orbifolds. One should also mention
the papers \cite{Aurich-Marklof96,Aurich-Steiner01,Marklof96}, where
some properties of quantum systems on hyperbolic orbifolds were
studied.

The paper is organized as follows. Section~\ref{s:prelim} contains
some background information on orbifolds. Section~\ref{s:classical}
is devoted to classical dynamical systems on orbifolds. Here we
prove Theorem~\ref{t:ergodic} on ergodicity of geodesic flows. In
Section~\ref{s:quantum} we prove Theorems~\ref{t:main1} and
\ref{t:main2}. Section~\ref{s:local-Weyl} contains the proof of the
local Weyl law for elliptic operators on orbifolds, which we use in
Section~\ref{s:quantum}.

\section{Preliminaries}\label{s:prelim}
In this section we briefly review some basic notions and results
concerning orbifolds. For more details on orbifold theory we refer
the reader to \cite{Adem-Leida-Ruan}.

\subsection{Orbifolds}
Let $X$ be a Hausdorff topological space. An $n$-dimensional
orbifold chart on $X$ is given by a triple $(\tilde{U},G_U,\phi_U)$,
where $\tilde{U}\subset {\mathbb R}^n$ is a connected open subset,
$G_U$ is a finite group acting on $\tilde{U}$ smoothly and $\phi_U :
\tilde{U} \to X$ is a continuous map, which is $G_U$-equinvariant
($\phi_U\circ g=\phi_U$ for all $g\in G_U$) and induces a
homeomorphism of $\tilde{U}/G_U$ onto an open subset
$U=\phi_U(\tilde{U})\subset X$. An embedding $\lambda :
(\tilde{U},G_U,\phi_U)\to (\tilde{V},G_V,\phi_V)$ between two
orbifold charts is a smooth embedding $\lambda : \tilde{U}\to
\tilde{V}$ with $\phi_V\circ\lambda=\phi_U$.

An orbifold atlas on $X$ is a family ${\mathcal
U}=\{(\tilde{U},G_U,\phi_U)\}$ of orbifold charts, which cover $X$
and are locally compatible: given any two charts
$(\tilde{U},G_U,\phi_U)$ over $\phi_U(\tilde{U})=U\subset X$ and
$(\tilde{U},G_V,\phi_V)$ over $\phi_V(\tilde{V})=V\subset X$, and a
point $x\in U\cap V$, there exists an open neighborhood $W$ of $x$
and a chart $(\tilde{W},G_W,\phi_W)$ over $W$ such that there are
embeddings $\lambda_U:(\tilde{W},G_W,\phi_W)\hookrightarrow
(\tilde{U},G_U,\phi_U)$ and
$\lambda_V:(\tilde{W},G_W,\phi_W)\hookrightarrow
(\tilde{V},G_V,\phi_V)$.

An (effective) orbifold $X$ of dimension $n$ is a paracompact
Hausdorff topological space equipped with an equivalence class of
$n$-dimensional orbifold atlases.

Throughout in the paper, $X$ will denote a compact orbifold.

Let $x\in X$ and let $(\tilde{U},G_U,\phi_U)$ be an orbifold chart
such that $x\in U=\phi_U(\tilde{U})$. Take any $\tilde{x}\in
\tilde{U}$ such that $\phi_U(\tilde{x})=x$. Let
$G_{\tilde{x}}\subset G_U$ be the isotropy group for $\tilde{x}$. Up
to conjugation, this group doesn't depend on the choice of chart and
will be called the local group at $x$. For any $x\in X$, there
exists an orbifold chart $(\tilde{U},G_U,\phi_U)$ such that $x\in
U=\phi_U(\tilde{U})$ and $G_U$ coincides with the local group
$G_{\tilde{x}}$ at $x$. Such an orbifold chart is called a
fundamental orbifold chart in a neighborhood of $x$.

A function $f:X\to {\mathbb C}$ is smooth iff for any orbifold chart
$(\tilde{U},G_U,\phi_U)$ the composition $f\left|_U\right.\circ
\phi_U$ is a smooth function on $\tilde{U}$. Denote by $C^\infty(X)$
the space of smooth functions on $X$.

\subsection{The cotangent bundle}
The cotangent bundle $T^*X$ of $X$ is an orbifold whose atlas is
constructed as follows. Let $(\tilde{U},G_U,\phi_U)$ is an orbifold
chart over $U\subset X$. Consider the local cotangent bundle
$T^*\tilde{U}=\tilde{U}\times {\mathbb R}^n$. It is equipped with a
natural action of the group $G_U$. The projection map
$T^*\tilde{U}\to \tilde{U}$ is $G_U$-equivariant, so we obtain a map
$p_U:T^*U:=T^*\tilde{U}/G_U\to U$, whose fiber $p_U^{-1}(x)$ is
homeomorphic to ${\mathbb R}^n/G_{\tilde{x}}$. $T^*X$ is obtained by
gluing together the bundles $p_U: T^*U\to U$ defined for each chart
$U$ in the atlas of $X$. Namely, let $(\tilde{U},G_U,\phi_U)$ and
$(\tilde{V},G_V,\phi_V)$ be two orbifold charts over
$\phi_U(\tilde{U})=U\subset X$ and $\phi_V(\tilde{V})=V\subset X$
respectively, and let $x$ belong to $U\cap V$. By definition, there
exist an open neighborhood $W$ of $x$ and a chart
$(\tilde{W},G_W,\phi_W)$ over $W$ such that there are embeddings
$\lambda_U:(\tilde{W},G_W,\phi_W)\hookrightarrow
(\tilde{U},G_U,\phi_U)$ and
$\lambda_V:(\tilde{W},G_W,\phi_W)\hookrightarrow
(\tilde{V},G_V,\phi_V)$. These embeddings give rise to
diffeomorphisms $\lambda_U: \tilde{W}\to \lambda_U(\tilde{W})\subset
\tilde{U}$ and $\lambda_V: \tilde{W}\to \lambda_V(\tilde{W})\subset
\tilde{V}$, which provide an equivariant diffeomorphism
$\lambda_{UV}=\lambda_V\lambda_U^{-1} : \lambda_U(\tilde{W})\to
\lambda_V(\tilde{W})$, the transition function. There are induced
equivariant embeddings of cotangent bundles $T^*\lambda_U:
T^*\tilde{W}\to T^*\tilde{U}$ è $T^*\lambda_V : T^*\tilde{W}\to
T^*\tilde{V}$. The local bundles $p_U:T^*\tilde{U}/G_U\to U$ è
$p_V:T^*\tilde{V}/G_V\to V$ are glued together by the transition
functions $T^*\lambda_{VU} =T^*\lambda_V(T^*\lambda_U)^{-1} :
T^*\lambda_U(T^*\tilde{W})\to T^*\lambda_V(T^*\tilde{W})$. Each
orbifold chart $(\tilde{U},G_U,\phi_U)$ on $X$ give rise to an
orbifold chart $(T^*\tilde{U},G_U,T^*\phi_U)$ on $T^*X$, where the
map $T^*\phi_U$ is induced by the projection $T^*\tilde{U}\to T^*U$.

Like in the manifold case, the cotangent bundle $T^*X$ carries a
canonical symplectic structure. Here by a symplectic form on an
orbifold $Y$ we mean an orbifold atlas ${\mathcal
U}=\{(\tilde{U},G_U,\phi_U)\}$ together with a $G_U$-invariant
symplectic form $\omega_U$ on $\tilde{U}$ for each
$(\tilde{U},G_U,\phi_U)\in {\mathcal U}$ such that, for any
transition function $\lambda_{UV} : \lambda_U(\tilde{W})\subset
\tilde{U}\to \lambda_V(\tilde{W})\subset\tilde{V}$, we have
$\lambda_{UV}^*\omega_V=\omega_U$. An orbifold $Y$ equipped with a
symplectic form $\omega$ is called a symplectic orbifold. The
canonical symplectic structure on $T^*X$ can be constructed as
follows. Consider the orbifold chart $(T^*\tilde{U},G_U,T^*\phi_U)$
induced by an orbifold chart $(\tilde{U},G_U,\phi_U)$.
$T^*\tilde{U}$ carries a canonical symplectic form
$\omega_{T^*\tilde{U}}$, which is invariant with respect to the
$G_U$-action on $T^*\tilde{U}$. These symplectic forms are
compatible for two different orbifold charts and define a symplectic
form on $T^*X$.

\subsection{Quotient presentations}
We will need the following well-known fact from orbifold theory due
to Kawasaki \cite{Kawasaki78,Kawasaki81} (see, for instance,
\cite{Bucicovschi,Moerdijk-Mrcun} for a detailed proof).

\begin{proposition}\label{p:quotient}
Let $M$ be a smooth manifold and $K$ a compact Lie group acting on
$M$ with finite isotropy groups. Then the quotient $X = M/K$ (with
the quotient topology) has a natural orbifold structure. Conversely,
any orbifold is a quotient of this type.
\end{proposition}

Any representation of an orbifold $X$ as the quotient $X \cong M/K$
of an action of a compact Lie group $K$ on a smooth manifold $M$
with finite isotropy groups will be called a quotient presentation
for $X$. There is a classical example of a quotient presentation for
an orbifold $X$ due to Satake. Choose a Riemannian metric on $X$. It
can be shown that the orthonormal frame bundle $M=F(X)$ of the
Riemannian orbifold $X$ is a smooth manifold, the group $K=O(n)$
acts smoothly, effectively and locally freely on $M$, and $M/K\cong
X$.

A quotient presentation $X \cong M/K$ for the orbifold $X$ gives
rise to a quotient presentation for the cotangent bundle $T^*X$ of
$X$ in the following way. The action of $K$ on $M$ induces an action
of $K$ on the cotangent bundle $T^*M$. Denote by $\mathfrak k$ the
Lie algebra of $K$. For any $v\in \mathfrak k$, denote by $v_M$ the
corresponding infinitesimal generator of the $K$-action on $M$. For
any $x\in M$, vectors of the form $v_M(x)$ with $v\in\mathfrak k$
span the tangent space $T_x(Kx)$ to the $K$-orbit of $x$. Denote
\[
(T^*_KM)_x  =\{\xi\in T^*_xM :\langle \xi, v_M(x)\rangle =0 \
\text{for any}\ v\in \mathfrak k\}.
\]
Since the action is locally free, the disjoint union
\[
T^*_KM=\bigsqcup_{x\in M}(T^*_KM)_x
\]
is a subbundle of the cotangent bundle $T^*M$, called the conormal
bundle. The bundle $T^*_KM$ is a $K$-invariant submanifold of $T^*M$
such that
\[
T^*_KM/K \cong T^*X,
\]
thus giving a quotient presentation for $T^*X$. This construction is
a particular case of the Marsden-Weinstein symplectic reduction (see
\cite{orbi-Egorov} for details).

\subsection{Pseudodifferential operators}
Here we recall some basic facts about pseudodifferential operators
on orbifolds (see \cite{Bucicovschi,Girbau-Nicolau1,Girbau-Nicolau2}
for details). As above, let $X$ be a compact orbifold.

A linear mapping $P : C^\infty(X) \to C^\infty(X)$ is a (pseudo)
differential operator on $X$ of order $m$ iff:

(1) the Schwartz kernel of $P$ is smooth outside of any neighborhood
of the diagonal in $X\times X$.

(2) for any $x \in X$ and for any orbifold chart $(\tilde{U}, G_U,
\phi_U)$ with $x \in U$, the operator $C^\infty_c (U)\ni f \mapsto
P(f)\left|_U\right.\in C^\infty(U)$ is given by the restriction to
$G_U$-invariant functions on $\tilde{U}$ of a (pseudo)differential
operator $\tilde{P}$ of order $m$ on $\tilde{U}$ that commutes with
the $G_U$ action.

All our pseudodifferential operators are assumed to be classical (or
polyhomogeneous), that is, their complete symbols can be represented
as an asymptotic sum of homogeneous functions. Denote by $\Psi^m(X)$
the class of pseudodifferential operators on $X$ of order $m$.

It is not hard to show \cite[Proposition 3.3]{Bucicovschi} that the
operator $\tilde{P}$ introduced in the part (2) of the definition is
unique up to a smoothing operator. In particular, it is unique if
$P$ is a differential operator. A pseudodifferential operator
$\tilde{P}$ on $\tilde{U}$ that commutes with the $G_U$-action has a
principal symbol $\tilde{p}\in C^\infty(T^*\tilde{U} \setminus
\{0\})$ that is invariant with respect to the $G_U$-action on
$T^*\tilde{U}$, and therefore it induces a function on the quotient
$T^*\tilde{U} \setminus \{0\}/G_U=T^*U \setminus \{0\}$. One can
check that these functions define a function on $T^*X \setminus
\{0\}$, the principal symbol of $P$. The pseudodifferential operator
$P$ on $X$ is elliptic if $\tilde{P}$ is elliptic for all choices of
orbifold charts.

\section{Classical dynamics on orbifolds}\label{s:classical}
\subsection{Hamiltonian dynamics}
The flow $F_t$ on a symplectic orbifold $(Y,\omega)$ is Hamiltonian
with a Hamiltonian $H\in C^\infty(Y)$ if, in any orbifold chart
$(\tilde{U},G_U,\phi_U)$, the infinitesimal generator $X_H$ of the
flow satisfies the standard relation
\[
i(X_H)\omega_U=d(H\left|_U\right.\circ \phi_U).
\]
Using quotient presentations, one can show the existence and
uniqueness of the Hamiltonian flow on a compact orbifold with an
arbitrary Hamiltonian $H$ (cf. \cite{Sjamaar-Lerman}). More
precisely, we have the following statement.

\begin{thm}\label{t:flow}
Let $v$ be a smooth vector field on a compact orbifold $Y$. Then
there exists a one-parameter group $\{\phi_t\}$ of diffeomorphisms
of $Y$ generated by $v$: for any $f\in C^\infty(Y)$
\[
\frac{d}{dt}f(\phi_t(x))\left|_{t=0}\right.=v(f)(x), \quad x\in Y.
\]
\end{thm}

\begin{proof}
Let $Y \cong N/K$ be a quotient presentation for the orbifold $Y$
and $\pi :N\to Y$ the corresponding projection map. Let $\mathcal H$
be a $K$-invariant distribution on $N$ such that ${\mathcal
H}_x\oplus T_x(Kx)=T_xN$ for any $x\in N$ (a horizontal
distribution). For instance, we can choose a $K$-invariant
Riemannian metric on $N$ and take $\mathcal H_x$ to be the
orthogonal complement of $T_x(Kx)$ in $T_xN$, $x\in N$. For any
$x\in N$, there exists a unique vector $\tilde{v}(x)\in {\mathcal
H}_x$ such that $d\pi_x(\tilde{v}(x))=v(\pi(x))$. Let
$\tilde{\phi}_t$ be the one-parameter group of diffeomorphisms of
$N$ generated by $\tilde{v}$. Each $\tilde{\phi}_t$ is a
$K$-invariant diffeomorphism of $N$ and therefore induces a
diffeomorphism $\phi_t$ of $Y$. It is easy to see that $\{\phi_t\}$
is the one-parameter group of diffeomorphisms of $Y$ generated by
$v$.
\end{proof}

The existence of the Hamiltonian flow $f_t$ on the cotangent bundle
$T^*X$ of a compact orbifold $X$ with Hamiltonian $p\in S^1(T^*X)$
follows immediately from Theorem~\ref{t:flow} applied to the
corresponding Hamiltonian vector field $v$ on an arbitrary level set
$\{(x,\xi)\in T^*X : p(x,\xi)=E\}, E>0$. In this case, one can give
another construction of the corresponding Hamiltonian flow $f_t$
based on symplectic reduction. Let $X\cong M/K$ be a quotient
presentation for $X$. Consider a Hamiltonian $p\in C^\infty(T^*X)$
as a smooth $K$-invariant function on $T^*_KM$. Let $\tilde{p}\in
C^\infty(T^*M)^K$ be an extension of $p$ to a smooth $K$-invariant
function on $T^*M$. Let $\tilde{f}_t$ be the Hamiltonian flow of
$\tilde{p}$ on $T^*M$. Since $\tilde{p}$ is $K$-invariant, the flow
$\tilde{f}_t$ preserves the conormal bundle $T^*_KM$, and its
restriction to $T^*_KM$ (denoted also by $\tilde{f}_t$) commutes
with the $K$-action on $T^*_KM$. So the flow $\tilde{f}_t$ on
$T^*_KM$ induces a flow $f_t$ on the quotient $T^*_KM/K=T^*X$, which
is called the reduced flow. One can show that this flow is a
Hamiltonian flow on $T^*X$ with Hamiltonian $p$. We refer the reader
to \cite{Sjamaar-Lerman} for more information on Hamiltonian
dynamics on singular symplectic spaces.

\subsection{Proof of Theorem~\ref{t:ergodic}}
Let $X$ be a compact Riemannian orbifold of negative sectional
curvature and $f_t : S^*X\to S^*X$ the geodesic flow of $X$. Let
$X\cong M/K$ be a quotient presentation for $X$. It gives rise to
the quotient representation for $S^*X$:
\[
S^*X\cong S^*_KM/K,
\]
where $S^*_KM=\{(x,\xi)\in T^*_KM: |\xi|=1\}$. Let ${\mathcal H}$ be
a horizontal distribution on $S^*_KM$ and let $\tilde{f}_t$ be a
$K$-invariant flow on $S^*_KM$, which is the horizontal lift of the
flow $f_t$ to $S^*_KM$ as in the proof of Theorem~\ref{t:flow}.

\begin{lem}\label{l:ph}
The flow $\tilde{f}_t$ on $S^*_KM$ is a partially hyperbolic flow.
\end{lem}

We recall that the flow $\tilde{f}_t$ is a partially hyperbolic if
the diffeomorphism $\tilde{f}_1$ is partially hyperbolic, which
means that there are distributions $E^s$, $E^c$ and $E^u$ on
$S^*_KM$, which are invariant under the map $d\tilde{f}_1$:
\[
d\tilde{f}_{1,(x,\xi)}(E^\tau(x,\xi))=E^\tau(\tilde{f}_1(x,\xi)),
\quad \tau = s,c,u,
\]
such that, for any $(x,\xi)\in S^*_KM$
\[
T_{(x,\xi)}(S^*_KM)=E^s(x,\xi)\oplus E^c(x,\xi)\oplus E^u(x,\xi),
\]
and there exist $C>0$ and
\[
0<\lambda_1\leq \mu_1<\lambda_2\leq \mu_2 < \lambda_3\leq \mu_3,
\quad \mu_1<1, \lambda_3>1,
\]
such that for $n>0$ we have
\[
C^{-1}\lambda_1^n\|v\|\leq \|d\tilde{f}_{n,(x,\xi)}v\|\leq
C\mu_1^n\|v\|, \quad v\in E^s(x,\xi),
\]
\[
C^{-1}\lambda_2^n\|v\|\leq \|d\tilde{f}_{n,(x,\xi)}v\|\leq
C\mu_2^n\|v\|, \quad v\in E^c(x,\xi),
\]
\[
C^{-1}\lambda_3^n\|v\|\leq \|d\tilde{f}_{n,(x,\xi)}v\|\leq
C\mu_3^n\|v\|, \quad v\in E^u(x,\xi).
\]

\begin{proof}[Proof of Lemma~\ref{l:ph}]
Let $(\tilde{U},G_U,\phi_U)$ be an orbifold chart over $U\subset X$
and $(T^*\tilde{U},G_U,T^*\phi_U)$ the induced orbifold chart over
$T^*U\subset T^*X$. The lift of the flow $f_t$ to $S^*\tilde U$ is
the geodesic flow of the metric $g_{\tilde U}$ on $\tilde U$
(denoted also by $f_t$). Since $g_{\tilde U}$ has negative sectional
curvature, $f_t$ is an Anosov flow, that is, there are distributions
$E^s, E^u\subset T(S^*\tilde U)$ such that for any $(x,\xi)\in
S^*\tilde U$
\[
T_{(x,\xi)}(S^*\tilde U)=E^s(x,\xi)\oplus E^c(x,\xi)\oplus
E^u(x,\xi),
\]
where $E^c$ is the tangent bundle to the orbits of the flow $f_t$.
Moreover, there are constants $C>0$ and $\lambda$, $0<\lambda<1$,
such that for $t>0$ we have
\[
\|df_{t,(x,\xi)}v\|\leq C\lambda^t\|v\|, \quad v\in E^s(x,\xi),
\]
\[
\|df_{-t,(x,\xi)}v\|\leq C\lambda^t\|v\|, \quad v\in E^s(x,\xi).
\]
It follows from the definition that the distributions $E^s$ and
$E^u$ are continuous and invariant under the map $df_t$.

Constructing an appropriate slice for the $K$-action on $M$, one can
give the following local description of the projection map $p:M\to
X$ (see, for instance, \cite[Proposition 2.1]{Bucicovschi} for
details).

\begin{proposition}\label{p:slice}
For any $x\in X$, there exists an orbifold chart
$(\tilde{U},G_U,\phi_U)$ defined in a neighborhood $U\subset X$ of
$x$ such that there exists a $K$-equivariant diffeomorphism
\[
p^{-1}(U)\cong K\times_{G_U}\tilde{U}.
\]
\end{proposition}

Recall that, by definition, $K\times_{G_U}\tilde{U}=(K\times
\tilde{U})/G_U$, where $G_U$ acts on $K\times \tilde{U}$ by
\[
\gamma\cdot (k,y)=(k\gamma^{-1}, \gamma y),\quad k\in K, y\in
\tilde{U}, \gamma \in G_U,
\]
and the $K$-action on $K\times_{G_U}\tilde{U}$ is given by the left
translations on $K$.

We have the corresponding local description for the projection map
$p_{S^*X} : S^*_KM\to S^*X$ associated with the quotient
presentation of $S^*X$:
\[
p_{S^*X}^{-1}(S^*U)\cong K\times_{G_U}S^*\tilde{U}.
\]

Using the finite covering $K\times S^*\tilde{U}\to
K\times_{G_U}S^*\tilde{U}$, one can lift the flow ${\tilde f}_t$ to
a flow $F_t$ on $K\times S^*\tilde{U}$. We have a commutative
diagram
\[
  \begin{CD}
K\times S^*\tilde{U} @>F_t>> K\times S^*\tilde{U}\\
@VVV                   @VVV        \\
S^*\tilde{U} @>f_t>> S^*\tilde{U}
  \end{CD}
\]
Since the flow $F_t$ is $K$-invariant, it is a group extension over
$f_t$:
\[
F_t(k,(x,\xi))=(k\phi_t(x,\xi),f_t(x,\xi)), \quad k\in K, (x,\xi)\in
S^*\tilde{U},
\]
where $\phi_t : S^*\tilde{U}\to K$ is a smooth function on
$S^*\tilde{U}$ with values in $K$.

By \cite[Theorem 2.2]{Brin-Pesin74}, the flow $F_t$ on $K\times
S^*\tilde{U}$ is a partially hyperbolic flow, and, therefore, the
flow $\tilde{f}_t$ is partially hyperbolic as well.
\end{proof}

The flow $\tilde{f}_t$ has a smooth invariant measure. Locally, with
respect to an equivariant trivialization $p_{S^*X}^{-1}(S^*U)\cong
K\times_{G_U}S^*\tilde{U}$ over an orbifold chart
$(S^*\tilde{U},G_U,T^*\phi_U)$ defined in a open set $U$, this
measure is the product $dk\times \mu$, where $dk$ is a Haar measure
on $K$ and $\mu$ is the Liouville measure on $S^*\tilde{U}$.

For any $(x,\xi)\in S^*_KM$, the space $E^c(x,\xi)$ coincides with
the tangent space of the $K$-orbit of $(x,\xi)$. As shown by Brin,
the distributions $E^s$ and $E^u$ are H\"older. Moreover, the
distributions $E^s$ and $E^u$ are integrable, and their integral
manifolds form invariant continuous foliations $W^s$ and $W^u$ on
$S^*_KM$. By \cite{Brin-Pesin74} and \cite{Pugh-Shub72}, the
foliations $W^s$ and $W^u$ are transversely absolutely continuous
with bounded Jacobian.

To prove ergodicity of the flow $f_t$ on $S^*X$, we apply the
classical Hopf argument to $K$-invariant functions on $S^*_KM$. We
will follow a detailed exposition of the proof of the ergodicity of
the geodesic flow given in \cite{Brin}. Let $\phi$ be a
$f_t$-invariant measurable function on $S^*X$. Then $\phi\circ
p_{S^*X}$ is a $\tilde{f}_t$-invariant and $K$-invariant continuous
function on $S^*_KM$. By \cite[Proposition 2.6]{Brin}, $\phi \circ
p_{S^*X}$ is mod 0 constant on the leaves of $W^s$ and $W^u$, that
is, there are null sets $N_s$ and $N_u$ in $S^*_KM$ such that $\phi
\circ p_{T^*X}(y)=\phi \circ p_{S^*X}(x)$ for any $x,y\in
S^*_KM\setminus N_s$, $y\in W^s(x)$ and $\phi \circ p_{T^*X}(z)=\phi
\circ p_{S^*X}(x)$ for any $x,z\in S^*_KM\setminus N_u$, $y\in
W^u(x)$. Consider the foliation $W^{uo}$ whose leaves are the orbits
of the leaves of $W^u$ under the $K\times {\mathbb R}$-action given
by the $K$-action on $S^*_KM$ and the flow $\tilde{f}_t$. By
\cite[Lemma 3.13]{Brin}, this foliation is absolutely continuous. So
we have two transversal absolutely continuous foliations $W^s$ and
$W^{uo}$ of complimentary dimensions and a measurable function $\phi
\circ p_{S^*X}$, which is mod 0 constant on the leaves of $W^s$ and
mod 0 constant on the leaves of $W^{uo}$. By \cite[Proposition
3.12]{Brin}, we conclude that $\phi \circ p_{S^*X}$ is mod 0
constant on $S^*_KM$ and, therefore, $\phi $ is mod 0 constant on
$S^*X$, that completed the proof of Theorem~\ref{t:ergodic}.

\section{Quantum ergodicity}\label{s:quantum}
In this section, we prove Theorems~\ref{t:main1} and \ref{t:main2}.
For this, we apply an abstract approach to quantum ergodicity
developed in \cite{Zelditch96}. We need two results, which maybe of
independent interest.

The first result is the local Weyl law for elliptic operators on
orbifolds. Throughout in this section, we assume that $P$ is a
positive, first-order self-adjoint elliptic pseudodifferential
operator on $X$ with positive principal symbol $p\in S^1(T^*X)$. We
introduce the generalized eigenvalue distribution function of $P$,
setting for any $A\in \Psi^0(X)$
\[
N_A(\lambda)=\operatorname{tr}AE_\lambda=\sum_{\{j:\lambda_j\leq\lambda\}}(A\psi_j,\psi_j),
\quad \lambda \in {\mathbb R},
\]
where as above $\{\psi_j\}$ is an orthonormal basis of
eigenfunctions of $P$ with corresponding eigenvalues $\{\lambda_j\}$
and $E_\lambda$ is the spectral projection of $P$, corresponding to
the semi-axis $(-\infty,\lambda]$.

\begin{thm}\label{t:local-Weyl}
For any $A\in \Psi^0(X)$, we have as $\lambda \to +\infty$
\[
N_A(\lambda)=\frac{1}{(2\pi)^n}\left(\int_{S^*_pX}\sigma_A\,d\mu\right)
\lambda^n+O(\lambda^{n-1}).
\]
\end{thm}

\begin{rem}
As a consequence of Theorem~\ref{t:local-Weyl}, we obtain Weyl's law
on counting eigenvalues, first proved by Farsi \cite{Farsi2001} (see
also \cite{Stanhope-Uribe}): as $\lambda \to +\infty$
\[
N(\lambda)=\frac{1}{(2\pi)^n}vol(S^*_pX) \lambda^n+O(\lambda^{n-1}).
\]
\end{rem}

\begin{rem}
Theorem \ref{t:local-Weyl} can be also proved by studying the
principal term at $t=0$ of the distribution trace
$\operatorname{tr}Ae^{itP}$ as in \cite{Stanhope-Uribe}. In
particular, in \cite{Stanhope-Uribe}, a pointwise Weyl law is
proved: as $\lambda \to +\infty$ we have
\[
\sum_{\{j:\lambda_j\leq\lambda\}}|\psi_j(x)|^2=\frac{B(x)}{(2\pi)^n}|G_{x}|\lambda^n+O(\lambda^{n-1}),
\]
where $B(x)$ is the volume of the domain $\{ \xi \in
T^*_{\tilde{x}}\tilde{U} : \tilde{p}(\tilde{x},\xi)\leq 1\}$,
$(\tilde{U},G_x,\phi_U)$ is a fundamental coordinate chart about $x$
with $\phi_U(\tilde{x})=x$.
\end{rem}

The proof of Theorem~\ref{t:local-Weyl} will be given in
Section~\ref{s:local-Weyl}.

The second result is the Egorov theorem for orbifolds proved in
\cite{orbi-Egorov}. As above, let $f_t$ be the Hamiltonian flow of
$p$ on $T^*X$.

\begin{theorem}\label{t:egorov-classical}
For any pseudodifferential operator $A$ of order $0$ with the
principal symbol $\sigma_{A}\in S^0(T^*X)$, the operator
\[
A(t)=e^{itP}Ae^{-itP}
\]
is a pseudodifferential operator of order $0$. Moreover, its
principal symbol $\sigma_{A(t)}\in S^0(T^*X)$ is given by
\[
\sigma_{A(t)}(x,\xi)=\sigma_{A}(f_t(x,\xi)), \quad (x,\xi)\in
T^*X\setminus 0.
\]
\end{theorem}

Now let $\mathcal A$ be the $C^*$ closure of the algebra $\Psi^0(X)$
of zeroth order pseudodifferential operators on $X$ acting in
$L^2(X)$. Define the automorphisms $\alpha_t^P$ of $\mathcal A$ by
\[
\alpha_t^P(A)=e^{itP}Ae^{-itP}.
\]
Then we have a $C^*$ dynamical system $(\mathcal A, \mathbb R,
\alpha)$.

We set for $A\in \mathcal A$
\[
\omega_j(A)=\langle A\psi_j, \psi_j\rangle,\quad \omega_\lambda
=\frac{1}{N(\lambda)}\sum_{\lambda_j\leq\lambda}\omega_j.
\]
Each $\omega_j$ is a normal invariant ergodic state of $(\mathcal A,
\mathbb R, \alpha)$. By Theorem~\ref{t:local-Weyl},
$\omega_\lambda\to \omega$ weakly as $\lambda\to\infty$. Therefore,
the system $(\mathcal A, \mathbb R, \alpha)$ is a quantized
Gelfand-Naimark-Segal system in the sense of \cite{Zelditch96}. By
Theorem~\ref{t:egorov-classical}, the classical limit system is the
$C^*$-dynamical system $(C(S^*_pX), \mathbb R, f_t^*)$, where
$f_t^*$ is the flow on the commutative $C^*$-algebra $C(S^*_pX)$
induced by the bicharacteristic flow $f_t$ on $S^*_pX$. The
classical limit system is abelian, so $(\mathcal A, \mathbb R,
\alpha)$ is a quantized abelian system. Thus, Theorems~\ref{t:main1}
and \ref{t:main2} are direct consequences of Theorems 1 and 2 in
\cite{Zelditch96}, respectively.

\section{The local Weyl law}\label{s:local-Weyl}
The goal of this section is to prove the local Weyl law, Theorem
\ref{t:local-Weyl}. We will use an approach based on the generalized
zeta-functions of elliptic operators (cf. \cite{Co-M}). Observe that
the zeta functions of elliptic operators on orbifolds were studied
in detail in \cite{Bucicovschi}. We slightly extend the results of
\cite{Bucicovschi}, introducing the orbifold analogues of the
canonical Kontsevich-Vishik trace and the Wodzicki-Guillemin
noncommutative residue.

\subsection{The canonical trace}
First, we describe an orbifold analogue of the canonical trace on
pseudodifferential operators introduced by Kontsevich and Vishik
\cite{KV1,KV2}.

For any $\sigma\in C^{\infty}({\mathbb R}^n\backslash \{0\})$,
homogeneous of degree $n$, i.e. such that
$\sigma(\lambda\eta)=\lambda^n\sigma(\eta)$ for any $\eta\not= 0$
and $\lambda \in {\mathbb R}^*_+$, let
$$
S({\sigma})=\int_{|\eta|=1} \sigma(\eta)d\eta.
$$

For any function $\phi$ on ${\mathbb R}^n\backslash \{0\}$, let
$$
\phi_{\lambda}(\eta)=\lambda^n\phi(\lambda\eta),\quad \lambda>0,
\eta\in {\mathbb R}^n\backslash \{0\}.
$$

Recall the following fact on continuation of a homogeneous smooth
function on ${\mathbb R}^n\backslash \{0\}$ to a homogeneous
distribution in ${\mathbb R}^n$, see \cite{Ho}, Theorems 3.2.3 and
3.2.4.
\begin{lemma}
\label{hom} Let $\sigma\in C^{\infty}({\mathbb R}^n\backslash
\{0\})$ be homogeneous of order $d$ in $\eta \in {\mathbb R}^n$.
\medskip
\par
\noindent (1)\ If $d\not \in \{-n-k:k\in {\mathbb N}\}$, $\sigma$
extends to a homogeneous distribution $\tau$ on ${\mathbb R}^n$.
\medskip
\par
\noindent (2)\ If $n=-q-k$, there is an extension $\tau$ of
$\sigma$, satisfying the condition
$$
(\tau,\phi)=\lambda^{-n-k}(\tau,\phi_{\lambda})+ \log \lambda
\sum_{|\alpha|=k} S(\eta^{\alpha}\sigma)
\partial_{\eta}^{\alpha}\phi(0)/\alpha!, \alpha>0.
$$
In particular, the obstruction to the existence of an extension
$\tau\in {\mathcal D}'({\mathbb R}^n)$, homogeneous in $\eta$, is
given by $S(\eta^{\alpha}\sigma)$, $|\alpha|=k$.
\end{lemma}

Let $L$ be the functional given by
$$
L(\sigma)=(2\pi)^{-n}\int_{{\mathbb R}^n} \sigma(\eta)d\eta,
$$
which is well-defined on symbols $\sigma\in S^m({\mathbb R}^n)$ of
order $m<-n$.

\begin{lemma}[\cite{KV1}]
The functional $L$ has an unique holomorphic extension $\tilde{L}$
to the space of classical symbols $S^z({\mathbb R}^n)$ of
non-integral order $z$. The value of $\tilde{L}$ on a symbol
$\sigma\sim \sum \sigma_{z-j}$ is given by
$$
\tilde{L}(\sigma)= (2\pi)^{-n}\int_{{\mathbb R}^n} (\sigma
-\sum_{j=0}^N \tau_{z-j}) d\eta,
$$
where $\tau_{z-j}$ is the unique homogeneous extension of
$\sigma_{z-j}$, given by Lemma~\ref{hom}, $N\geq {\rm Re}\,z +n$.
\end{lemma}

Let $(\tilde{U}, G_U, \phi_U)$ be a fundamental orbifold chart on
$X$ and $A\in\Psi^{m}(\tilde{U})$, $m<-n$, is a $G_U$-invariant
pseudodifferential operator on $\tilde{U}$ with complete symbol
$k\in S^m(\tilde{U}\times {\mathbb R}^n)$. Its trace is given by the
formula
$$
\tr(A)=(2\pi)^{-n}\frac{1}{|G_U|}\int_{\tilde{U}\times {\mathbb
R}^n} k(x,\xi)\, dx\, d\xi.
$$
The following formula provides an extension of the trace functional
to any pseudodifferential operator $A\in \Psi^{z}(\tilde{U})$ of
arbitrary non-integral order $z\in {\mathbb C}\backslash {\mathbb
Z}$:
\[
\TR(A)=\frac{1}{|G_U|} \int_{\tilde{U}} \tilde{L}(k(x,\cdot))\,dx\,.
\]
This definition can be extended to all operators $P\in \Psi^{z}(X),
z\in {\mathbb C}\backslash {\mathbb Z}$.

Using \cite{KV1,KV2}, we immediately obtain the following statement.

\begin{prop}
The linear functional $\TR$ on the class $\Psi^{m+{\mathbb Z}}(X),
m\in {\mathbb C}\backslash {\mathbb Z}$ of classical
pseudodifferential operators of orders $\alpha\in m+{\mathbb Z}$ has
the following properties:
\medskip
\par
\noindent (1) It coincides with the usual trace $\tr$ for ${\rm
Re}\, \alpha<-n$.
\medskip
\par
\noindent (2) It is a trace functional, i.e. $\TR([A,B])=0$ for any
$A\in \Psi^{\alpha_1}(X)$ and $B\in \Psi^{\alpha_2}(X)$,
$\alpha_1+\alpha_2\in m+{\mathbb Z}$.
\end{prop}

\subsection{The noncommutative residue}
Now let us turn to an orbifold analogue of the Wodzicki-Guillemin
noncommutative residue \cite{Gu85,Wo}.

Let $(\tilde{U}, G_U, \phi_U)$ be a fundamental orbifold chart on
$X$. For a $G_U$-invariant pseudodifferential operator
$A\in\Psi^{m}(\tilde{U})$, we define its residue form $\rho_A$ on
${\tilde{U}\times {\mathbb R}^n}$ as
\[
\rho_A = k_{-n}(x,\xi)\, dx\,d\xi,
\]
and its residue trace $\tau(A)$ as
\[
\tau(A)=\frac{1}{|G_U|} \int_{|\xi|=1} k_{-n}(x,\xi)\, dx\,d\xi.
\]
It can be easily checked that, for any $A\in \Psi^{m}(X)$, its
locally defined residue forms $\rho_A$ give rise to a well-defined
volume form on $$T^*X\setminus \{0\}$$, and the residue trace
$\tau(A)$ is given by integration of the residue form $\rho_A$ over
$S^*_pX$:
\[
\tau(A)=\int_{S^*_pX}\rho_A.
\]

Now we describe a relation between the canonical trace and the
noncommutative residue. First, recall that a family $A(z)\in
\Psi^{f(z),-\infty}(\tilde{U})$ is holomorphic (in a domain
$D\subset {\mathbb C}$), if the order $f(z)$ is a holomorphic
function, and the complete symbol $k(z)\in
S^{f(z),-\infty}(\tilde{U}\times {\mathbb R}^{n})$ of $A(z)$ is
represented as an asymptotic sum
$$
k(z,x,\xi)\sim \sum_{j=0}^{\infty} \theta(\xi) k_{f(z)-j}(z,x,\xi),
$$
which is uniform in $z$, with homogeneous components
$k_{f(z)-j}(z,x,\xi)$, holomorphic in $z$.

\begin{prop}
\label{p:extension} For any holomorphic family $A(z)\in \Psi^{m +
z}(X), z\in D\subset {\mathbb C}$, the function $z\mapsto \TR(A(z))$
is meromorphic with no more than simple poles at $z_k=-m-n+k\in
D\bigcap {\mathbb Z},k\geq 0$ and with
\[
\res_{z=z_k} \TR(A(z))=\tau(A(z_k)).
\]
\end{prop}
\begin{proof} The proposition is an immediate consequence of the similar
fact for usual pseudodifferential operators \cite{KV1,KV2} (see also
\cite{Bucicovschi}).
\end{proof}

\subsection{Proof of Theorem \ref{t:local-Weyl}}
Let $P$ be a positive, first-order self-adjoint elliptic
pseudodifferential operator on $X$ with positive principal symbol
$p\in S^1(T^*X)$ and $A\in \Psi^{0}(X)$. Then we have a holomorphic
family $A(z)=AP^z\in \Psi^{z}(X), z\in {\mathbb C}$ (see
\cite{Bucicovschi}). By Proposition \ref{p:extension}, we obtain
that the function $z\mapsto \TR(AP^z)$ is meromorphic with no more
than simple poles at $z_k=-n+k,k\geq 0$. In particular, for ${\rm
Re}\,z< -n$, the function $z\mapsto \TR(AP^z)$ is holomorphic and
\[
\TR(AP^z)=\Tr(AP^z).
\]
Therefore, $\TR(AP^z)$ gives a meromorphic extension of the
generalized zeta function $\zeta_A(z)=\Tr(AP^z)$ to the complex
plane. Moreover, by Proposition \ref{p:extension}, we have
\[
\res_{z=-n} \TR(AP^z)=\tau(AP^{-n}).
\]
By definition, the residue form $\rho_{AP^{-n}}$ is given by
\[
\rho_{AP^{-n}}=d_{-n}(x,\xi)\, dx\,d\xi,
\]
where $d_{-n}$ is the homogeneous component of degree $-n$ of the
complete symbol of the operator $AP^{-n}$. Since $AP^{-n}\in
\Psi^{-n}(X)$, $d_{-n}$ coincides with the principal symbol of
$AP^{-n}$. Therefore $d_{-n}$ is invariantly defined as a function
on $T^*X\setminus \{0\}$ and equals
\[
d_{-n}=\sigma_A p^{-n}.
\]
The proof is completed by using the formula
\[
\zeta_A(z)=\int_0^{+\infty}\lambda^z\,d_\lambda N_A(\lambda)
\]
and the Ikehara Tauberian theorem (see, for instance,
\cite{Shubin:pdo}).


\begin{thebibliography}{00}
\bibitem{Adem-Leida-Ruan}
A. Adem, J. Leida, Y. Ruan, \textit{Orbifolds and stringy topology.}
Cambridge Tracts in Mathematics, 171. Cambridge University Press,
Cambridge, 2007.

\bibitem{Anosov67}
D.V. Anosov, \textit{Geodesic flows on closed Riemannian manifolds
of negative curvature.} Trudy Mat. Inst. Steklov. \textbf{90}, 1967.

\bibitem{Aurich-Marklof96}
R. Aurich, J. Marklof, \textit{Trace formulae for three-dimensional
hyperbolic lattices and application to a strongly chaotic
tetrahedral billiard.} Phys. D \textbf{92} (1996), 101--129.

\bibitem{Aurich-Steiner01}
R. Aurich, F. Steiner, \textit{Orbit sum rules for the quantum wave
functions of the strongly chaotic Hadamard billiard in arbitrary
dimensions.} Invited papers dedicated to Martin C. Gutzwiller, Part
V. Found. Phys. \textbf{31} (2001), 569--592.

\bibitem{Bucicovschi}
B. Bucicovschi, \textit{Seeley's theory of pseudodifferential
operators on orbifolds.} Preprint arXiv:math/9912228.

\bibitem{Brin}
M. Brin, \textit{Ergodicity of the geodesic flow.} Appendix to {\em
Lectures on spaces of nonpositive curvature}, by W. Ballmann, DMV
Seminar, 25 Birkh\"auser, Basel, 1995.

\bibitem{Brin-Pesin74}
M. I. Brin, Ja. B. Pesin, \textit{Partially hyperbolic dynamical
systems.} Izv. Akad. Nauk SSSR Ser. Mat. \textbf{38} (1974),
170--212.

\bibitem{CdV} Y. Colin de Verdi\`ere, \textit{Ergodicit\'e et fonctions propres
du laplacien.} Comm. Math. Phys. \textbf{102} (1985), 497--502.

\bibitem{Co-M} A. Connes, H. Moscovici, \textit{The local index
formula in noncommutative geometry.} Geom. and Funct. Anal.
\textbf{5} (1995), 174--243.

\bibitem{Dryden-G-G-W}
E. B. Dryden, C. S. Gordon, S. J. Greenwald, D. L. Webb,
\textit{Asymptotic expansion of the heat kernel for orbifolds.}
Michigan Math. J. \textbf{56} (2008), 205--238.

\bibitem{Farsi2001}
C. Farsi, \textit{Orbifold spectral theory.} Rocky Mountain J. Math.
\textbf{31} (2001), 215--235.

\bibitem{Girbau-Nicolau1}
J. Girbau, M. Nicolau, \textit{Pseudodifferential operators on
$V$-manifolds and foliations. I.} Collect. Math. \textbf{30} (1979),
247--265

\bibitem{Girbau-Nicolau2}
J. Girbau, M. Nicolau, \textit{Pseudodifferential operators on
$V$-manifolds and foliations. II.} Collect. Math. \textbf{31}
(1980), 63--95

\bibitem{Gu85} V. Guillemin, \textit{A new proof of Weyl's formula on the
asymptotic distribution of eigenvalues.} Adv. Math. \textbf{55}
(1985), 131--160.

\bibitem{Guilemin-Uribe-Wang}
V. Guillemin, A. Uribe, Z. Wang, \textit{Geodesics on weighted
projective spaces.} Ann. Global Anal. Geom. \textbf{36} (2009),
205--220.

\bibitem{Ho} L. H\"ormander, \textit{The analysis of linear
partial differential operators I.} Berlin Heidelberg New York Tokyo:
Springer 1983.

\bibitem{Ja-Str06}
D. Jakobson, A. Strohmaier, \textit{High energy limits of
Laplace-type and Dirac-type eigenfunctions and frame flows.} Comm.
Math. Phys. \textbf{270} (2007), 813--833.

\bibitem{Kawasaki78} T. Kawasaki, \textit{The signature theorem for
$V$-manifolds.} Topology \textbf{17} (1978), 75--83

\bibitem{Kawasaki81} T. Kawasaki, \textit{The index of elliptic operators
over $V$-manifolds.} Nagoya Math. J. \textbf{84} (1981), 135--157

\bibitem{KV1} M. Kontsevich, S. Vishik, \textit{Determinants
of elliptic pseudo-differential operators.} Preprint MPI/94-30,
1994, 156pp.

\bibitem{KV2} M. Kontsevich, S. Vishik, \textit{Geometry of  determinants
of elliptic operators.} Functional analysis on the eve of the 21st
century. Vol. I. (Progress in Mathematics. Vol. 132) Boston:
Birkh\"auser 1996. - 173 - 197.

\bibitem{orbi-Egorov}
Yu. A. Kordyukov, \textit{Classical and quantum dynamics on
orbifolds.} SIGMA Symmetry Integrability Geom. Methods Appl.
\textbf{7} (2011), 106, 12 pages

\bibitem{Marklof96}
J. Marklof, \textit{On multiplicities in length spectra of
arithmetic hyperbolic three-orbifolds.} Nonlinearity \textbf{9}
(1996), 517--536.

\bibitem{Moerdijk-Mrcun} I, Moerdijk, J. Mr{\v{c}}un,
\textit{Introduction to foliations and Lie groupoids.} Cambridge
Studies in Advanced Mathematics, 91. Cambridge University Press,
Cambridge, 2003.

\bibitem{Pugh-Shub72}
Ch. Pugh, M. Shub, \textit{Ergodicity of Anosov actions.} Invent.
Math. \textbf{15} (1972), 1--23.

\bibitem{Sh} A. I. Shnirelman, \textit{Ergodic properties of eigenfunctions.}
Uspehi Mat. Nauk \textbf{29} (1974), no. 6(180), 181--182.

\bibitem{Shubin:pdo}
M. A. Shubin, \textit{Pseudodifferential operators and spectral
theory}. Springer-Verlag, Berlin, 2001.


\bibitem{Sjamaar-Lerman} R. Sjamaar, E. Lerman, \textit{Stratified symplectic spaces and
reduction.} Ann. of Math. (2) \textbf{134} (1991), 375--422.

\bibitem{Stanhope-Uribe}
E. Stanhope, A. Uribe, \textit{The spectral function of a Riemannian
orbifold.} Ann. Global Anal. Geom. \textbf{40} (2011), 47--65

\bibitem{Wo} M. Wodzicki, \textit{Noncommutative residue. Part I.
Fundamentals.} K-theory, arithmetic and geometry (Moscow, 1984-86),
Lecture Notes in Math. 1289, pp. 320--399. Berlin Heidelberg New
York: Springer 1987.

\bibitem{Zelditch87} S. Zelditch, \textit{Uniform distribution of
eigenfunctions on compact hyperbolic surfaces.} Duke Math. J.
\textbf{55} (1987), 919--941.

\bibitem{Zelditch96}
S. Zelditch, \textit{Quantum ergodicity of $C\sp *$ dynamical
systems.} Comm. Math. Phys. \textbf{177} (1996), 507--528.

\bibitem{Zelditch10}
S. Zelditch, \textit{Recent developments in mathematical quantum
chaos.} Current developments in mathematics, 2009, pp. 115--204,
Int. Press, Somerville, MA, 2010.
\end{thebibliography}
\end{document}